\documentclass[preprint]{elsarticle}
\usepackage{graphicx, epstopdf}
\usepackage{geometry}
\usepackage{amssymb,amsmath}
\usepackage[toc,page,title,titletoc,header]{appendix}
\usepackage{amsbsy}
\usepackage{amsthm}
\usepackage{multirow}
\usepackage{indentfirst}
\usepackage{natbib}
\usepackage{amscd}
\numberwithin{equation}{section}

\newtheorem{Theorem}{Theorem}[section]

\newtheorem{Lemma}{Lemma}[section]

\newtheorem{Assumption-Notation}[Theorem]{Assumption-Notation}

\newtheorem{Conjecture}{Conjecture}

\newtheorem{Remark}{Remark}[section]
\newtheorem{Corollary}{Corollary}[section]

\makeatletter
\renewcommand\@biblabel[1]{}
\renewenvironment{thebibliography}[1]
{\section*{\refname}%
\@mkboth{\MakeUppercase\refname}{\MakeUppercase\refname}%
\list{\@biblabel{\@arabic\c@enumiv}}%
{\settowidth\labelwidth{\@biblabel{#1}}%
\leftmargin\labelwidth
\advance\leftmargin\labelsep
\advance\leftmargin by 2em%
\itemindent -2em%
\@openbib@code
\usecounter{enumiv}%
\let\p@enumiv\@empty
\renewcommand\theenumiv{\@arabic\c@enumiv}}%
\sloppy
\clubpenalty4000
\@clubpenalty \clubpenalty
\widowpenalty4000%
\sfcode`\.\@m}
{\def\@noitemerr
{\@latex@warning{Empty `thebibliography' environment}}%
\endlist}
\makeatother

\long\def\symbolfootnote[#1]#2{\begingroup
\def\thefootnote{\fnsymbol{footnote}}\footnote[#1]{#2}\endgroup}

\begin{document}
\begin{frontmatter}
\title{Occupation times of refracted L\'evy processes with jumps having rational Laplace transforms}

 \author{Lan Wu\corref{}}
 \ead{lwu@pku.edu.cn}
 \author{Jiang Zhou\corref{cor1}}
 \ead{1101110056@pku.edu.cn}
 \cortext[cor1]{Corresponding author.}

 %\address{\large Jiang Zhou   \\School of Mathematics Sciences\\Peking University\\Beijing 100871\\P.R.China}
%\email{\large 1101110056@pku.edu.cn }
%\address{\large Lan Wu
% \\School of Mathematics Sciences\\Peking University\\Beijing 100871\\P.R.China}
%\email{\large lwu@pku.edu.cn}
\address{School of Mathematical Sciences, Peking University, Beijing 100871, P.R.China}

\begin{abstract}{We investigate a refracted L\'evy process driven by a jump diffusion process, whose jumps have rational Laplace transforms. For such a stochastic process, formulas for the Laplace transform of its occupation times are deduced. To derive the main results, some modifications on our previous approach have been made. In addition, we obtain a very interesting identity, which is conjectured to hold for a general refracted L\'evy process.
 }
\end{abstract}

\end{frontmatter}

\section{Introduction}
A refracted L\'evy process $\tilde{U}=(\tilde{U}_t)_{t\geq 0}$ (which is proposed in Kyprianou and Loeffen (2010)) is a stochastic process whose dynamic is given by
\begin{equation}
\tilde{U}_t=X_t-\alpha \int_{0}^{t}\textbf{1}_{\{\tilde{U}_s > b\}}ds,
\end{equation}
where $\alpha, b \in \mathbb R$, $X=(X_t)_{t\geq 0}$ is a L\'evy process and $\textbf{1}_{A}$ represents the indicator function. Under the assumption that $X$ in (1.1) is a L\'evy process without positive jumps, many results on the corresponding process $\tilde{U}$ have been derived, one can refer to Kyprianou and Loeffen (2010), Kyprianou et al. (2014) and Renaud (2014). For example, formulas for the joint Laplace transform of $(\nu_a^-, \int_{0}^{\nu_a^-}\textbf{1}_{\{\tilde{U}_s< b\}}ds)$ has been derived  in Renaud (2014), where $\nu_a^-$ is the first passage time of $\tilde{U}$. However, all the above three papers focus on the case that $X$ in (1.1) is a L\'evy  process with only negative jumps.

In Zhou and Wu (2015), we considered a slight different process $U=(U_t)_{t \geq 0}$ which satisfies
\begin{equation}
\begin{split}
&dU_t = dX_t - \alpha \textbf{1}_{\{U_t < b\}}dt, \ \ t>0,
\end{split}
\end{equation}
with $U_0 = X_0$, where $X=(X_t)_{t \geq 0}$ is a hyper-exponential jump diffusion process. In that paper, we have derived formulas for the Laplace transform of $\int_{0}^{t}\textbf{1}_{\{U_s < b\}}ds$ with respect to $t$.
Here, we continue this research under a more general assumption on $X$.
Note that if $\mathbb P_x(\tilde{U}_t=b) = 0$ for Lebesgue almost every $t \geq 0$, then the two processes $\tilde{U}$ and $U$ are the same in essence. This means that the process $U$ can be still treated as a refracted L\'evy process. So this paper and Zhou and Wu (2015) extend the research on refracted L\'evy processes to the case that the underlying L\'evy process $X$ has two-sided jumps.

Since the process $X$ considered here is quite general, some calculations become extremely complicated if we  apply directly the approach developed in Zhou and Wu (2015). To reduce the complexity, some modifications have been done on the method in Zhou and Wu (2015). With this updated method, we success in deriving formulas for the Laplace transform of $\int_{0}^{t}\textbf{1}_{\{U_s < b\}}ds$. Moreover, we obtain a very useful identity (see (4.1) below) and conjecture that this identity holds for a general refracted L\'evy process $U$ and the matching driven L\'evy process $X$. A future research direction is to show this interesting conjecture.

One immediate application of our results is to compute the total time of charging fees for Variable Annuities with a state-dependent fee structure as in Zhou and Wu (2015).
Such a fee charging method was proposed in Bernard et al. (2014). Under the state-dependent fee structure, the insurers deduct fees only when policyholders' account value is lower than a pre-specified level. This fee deducting approach has several advantages, see Bernard et al. (2014) and Delong (2014) for the details. The reason why we are interested in the quantity $\int_{0}^{t}\textbf{1}_{\{U_s < b\}}ds$ is that it is the total time that the insures can charge fees under the state-dependent fee structure; see Zhou and Wu (2015) for more details.

The remainder of this paper is organized as follows. Details of our model are introduced in Section 2, meanwhile, some important results are given. We then derive the main results in Section 3 and finally present an interesting conjecture in Section 4.

\section{Model assumption and preliminary results}
\subsection{The model.}
Let $X=(X_t)_{t \geq 0}$ be a jump diffusion process whose dynamic is given by
\begin{equation}
  X_t = X_0 + \mu t+\sigma W_t + \sum_{k=1}^{N^+_t}Y^+_k-\sum_{k=1}^{N^-_t}Y^-_k,
\end{equation}
where $\mu$ and $X_0$ are constants; $(W_t)_{t\geq 0}$ is a standard Brownian motion with $W_0=0$, and $\sigma > 0$ is the volatility of the diffusion; $(N^+_t)_{t\geq 0}$ is a Poisson process with rate $\lambda^+$, and $(N^-_t)_{t\geq 0}$ is a Poisson process with rate $\lambda^-$; $Y^+_k$ $\left(Y_k^-\right)$,  $k=1, 2,...$, are independent and identically distributed random variables; moreover, $(W_t)_{t\geq 0}$, $(N^+_t)_{t\geq 0}$, $(N^-_t)_{t\geq 0}$, $\{Y^+_k; k=1,2,\ldots\}$ and $\{Y^-_k; k=1,2,\ldots\}$ are independent mutually; finally, the density functions of $Y_1^+$ and  $Y_1^-$ are given respectively by
\begin{equation}
p^+(y)=\sum_{k=1}^{m^+}\sum_{j=1}^{m_k}c_{kj}\frac{(\eta_k)^jy^{j-1}}{(j-1)!}e^{-\eta_k y}, \ \ y > 0,
\end{equation}
and
\begin{equation}
p^-(y)=\sum_{k=1}^{n^-}\sum_{j=1}^{n_k}d_{kj}\frac{(\vartheta_k)^jy^{j-1}}{(j-1)!}e^{-\vartheta_k y}, \ \ y > 0.
\end{equation}
In the following, the law of $X$ starting from $x$ is denoted by $\mathbb P_x$ with $\mathbb E_x$ denoting the corresponding expectation; when $x = 0$, we write $\mathbb P$ and $\mathbb E$ for brevity. Without loss of generality, in (2.2) and (2.3), we assume that $\eta_i \neq \eta_j$ and $\vartheta_i \neq \vartheta_j$ for all $i \neq j$.

\begin{Remark}
The process $X$ defined by (2.1) is a L\'evy process with jumps having rational Laplace transforms and  has been investigated by many papers, see, e.g., Lewis and Mordecki (2008) and Fourati (2010).
The probability density function (2.2) is quite general, including phase-type distributions. So the process $X$ in (2.1) can be used to approximate any other L\'evy process (see Proposition 1 in Asmussen et al. (2004)).
\end{Remark}

For given $b \in \mathbb R$, let $U_t$ be the unique strong solution (which is a strong Markov process; see Remark 3 in Kyprianou and Loeffen (2010)) of the following stochastic differential equation:
\begin{equation}
\begin{split}
&dU_t = dX_t - \alpha \textbf{1}_{\{U_t < b\}}dt, \ \ t>0,\\
&and  \ \ \
U_0 = X_0,
\end{split}
\end{equation}
where $X$ is given by (2.1) and  $\alpha \in \mathbb R$ is fixed. In the following, we want to derive formulas for the Laplace transform of $\int_{0}^{t}\textbf{1}_{\{U_s < b\}}ds$, i.e.,
\begin{equation}
\int_{0}^{\infty}e^{-q t}\mathbb E_x\left[\int_{0}^{t}\textbf{1}_{\{U_s < b\}}ds\right]dt=\frac{1}{q}\mathbb E_x\left[\int_{0}^{e(q)}\textbf{1}_{\{U_s < b\}}ds\right],
\end{equation}
where $q > 0$ and $e(q)$ stands for an exponential random variable, which is independent of the process $U$ and whose expectation is equal to $\frac{1}{q}$.

For any given and fixed $\alpha \in \mathbb R$, set $Y = \{Y_t :=  X_t - \alpha t; t\geq 0\}$. Denote by $\tilde{\mathbb P}_y$ the law of $Y$ starting from $y$ and by $\tilde{\mathbb E}_y$ the corresponding expectation. When $y=0$, we write briefly $\tilde{\mathbb P}$ and
$\tilde{\mathbb E}$. For any $a, c \in \mathbb R$, we introduce the following stopping times of $X$, $Y$ and $U$:
\begin{equation}
\tau_{a,X}^{-} := \inf\{t\geq 0: X_t < a\}, \ \ \tau_{c,Y}^{+}:= \inf\{t\geq 0: Y_t \geq c \},
\end{equation}
\begin{equation}
\kappa_a^-:=\inf\{t\geq 0: U_t < a \}, \ \ \ \ \kappa_c^+ :=\inf\{t \geq 0: U_t \geq c\}.
\end{equation}

\subsection{Some preliminary results.}
First, introduce the following two rational functions:
\begin{equation}
\begin{split}
&\psi(z):= iz\mu
-\frac{\sigma^2}{2}z^2  +\lambda^+\left(\sum_{k=1}^{m^+}\sum_{j=1}^{m_k}\frac{c_{kj}(\eta_k)^j}{(\eta_k-iz)^j}-1\right)
+
\lambda^-\left(\sum_{k=1}^{n^-}\sum_{j=1}^{n_k}\frac{d_{kj}(\vartheta_k)^j}{(\vartheta_k+iz)^j}-1\right),\\
&\tilde{\psi}(z):= \psi(z) - i\alpha z.
\end{split}
\end{equation}
For $z\in \mathbb R$, we have $\tilde{\psi}(z)= \ln\left(\tilde{\mathbb E}\left[e^{iz Y_1}\right]\right)$ and $\psi(z)= \ln\left(\mathbb E\left[e^{iz X_1}\right]\right)$. Besides, in the rest of the paper, let $Re(x)$ represent the real part of $x$.

The following lemma is taken from  Lemma 1.1 in Lewis and Mordecki (2008).

\begin{Lemma}
For $q > 0$, the equation $\tilde{\psi}(z)=\psi(z)- i\alpha z = q$ has, in the set $Im(z) < 0$, a total of $M^+_q$ distinct roots $-i\beta_{1,q}$, $-i\beta_{2,q}$, $\ldots$, $-i\beta_{M^+_q,q}$, with respective multiplicities $M^q_{1}(=1), M^q_{2}, \ldots, M^q_{M^+_q}$, ordered such that $0< \beta_{1,q}< Re(\beta_{2,q})\leq \cdots \leq Re(\beta_{M^+_q,q})$. Moreover,
\begin{equation}
\sum_{k=1}^{M^+_q}M^q_{k}=\sum_{k=1}^{m^+}m_k+1.
\end{equation}
\end{Lemma}

The following lemma summarizes the results given by Theorem 2.2 and Corollary 2.1 in  Lewis and Mordecki (2008), but uses different notations.

\begin{Lemma}
 For any given $q > 0$, it holds that
\begin{equation}
\tilde{\mathbb E}\left[e^{-s \overline{Y}_{e(q)}}\right]
=\prod_{k=1}^{m^+}\left(\frac{s+\eta_k}{\eta_k}\right)^{m_k}
\prod_{k=1}^{M^+_q}\left(\frac{\beta_{k,q}}{s+\beta_{k,q}}\right)^{M^q_{k}} =\frac{C_1}{s+\beta_{1,q}}+\sum_{k=2}^{M_q^+}\sum_{j=1}^{M^q_{k}}\frac{C_{kj}}{(s+\beta_{k,q})^j}:=\psi^+(s), \ \ s \geq 0,
\end{equation}
and
\begin{equation}
\tilde{\mathbb P}\left(\overline{Y}_{e(q)}\in dy\right)=C_1 e^{-\beta_{1,q}y}dy+\sum_{k=2}^{M_q^+}
\sum_{j=1}^{M^q_{k}}C_{kj}\frac{y^{j-1}}{(j-1)!}e^{-\beta_{k,q}y}dy, \ \ y \geq 0,
\end{equation}
where
\begin{equation} \frac{C_1}{\beta_{1,q}}=\prod_{k=1}^{m^+}\left(\frac{\eta_k-\beta_{1,q}}{\eta_k}\right)^{m_k}\prod_{k=2}^{M_q^+}\left(
\frac{\beta_{k,q}}{\beta_{k,q}-\beta_{1,q}}\right)^{M^q_{k}},
\end{equation}
and for $k=2,\ldots, M_q^+$ and $j=0,\cdots, M^q_{k}-1$,
\begin{equation}
C_{k,M^q_{k}-j}=\frac{1}{j!}\left[\frac{\partial^j}{\partial s^j}\left(\psi^+(s)(s+\beta_{k,q})^{M^q_{k}}\right)\right]_
{s=-\beta_{k,q}}.
\end{equation}
\end{Lemma}

An application of the above two lemmas to $-X$ yields the following result.

\begin{Lemma}
(1) For $q > 0$, the equation $\psi(z) = q$ has, in the set $Im(z) > 0$, a total of $N_q^-$ distinct roots $i\gamma_{1,q}$, $i\gamma_{2,q}$, $\ldots$, $i\gamma_{N_q^-,q}$, with respective multiplicities $N^q_{1}(=1), N^q_{2}, \ldots, N^q_{N_q^-}$, ordered such that $0< \gamma_{1,q}< Re(\gamma_{2,q})\leq \cdots \leq Re(\gamma_{N_q^-,q})$. In addition,
\begin{equation}
\sum_{k=1}^{N_q^-}N^q_{k}=1+ \sum_{k=1}^{n^-}n_k.
\end{equation}

(2) For any given $q > 0$, we have
\begin{equation}
\mathbb E\left[e^{s\underline{X}_{e(q)}}\right]
=\prod_{k=1}^{n^-}\left(\frac{s+\vartheta_k}{\vartheta_k}\right)^{n_k}
\prod_{k=1}^{N_q^-}\left(\frac{\gamma_{k,q}}{s+\gamma_{k,q}}\right)^{N^q_{k}}
=\frac{D_1}{s+\gamma_{1,q}}+\sum_{k=2}^{N_q^-}\sum_{j=1}^{N^q_{k}}\frac{D_{kj}}{(s+\gamma_{k,q})^j}:=\psi^-(s), \ \ s \geq 0,
\end{equation}
and
\begin{equation}
\mathbb P\left(\underline{X}_{e(q)}\in dy\right)= D_1 e^{\gamma_{1,q}y}dy+\sum_{k=2}^{N_q^-}
\sum_{j=1}^{N^q_{k}}D_{kj}\frac{(-y)^{j-1}}{(j-1)!}e^{\gamma_{k,q}y}dy, \ \ y \leq 0,
\end{equation}
where
\begin{equation} \frac{D_1}{\gamma_{1,q}}=\prod_{k=1}^{n^-}\left(\frac{\vartheta_k-\gamma_{1,q}}{\vartheta_k}\right)^{n_k}\prod_{k=2}^{N_q^-}
\left(\frac{\gamma_{k,q}}{\gamma_{k,q}-\gamma_{1,q}}\right)^{N^q_{k}},
\end{equation}
and for $k=2,\ldots, N_q^-$ and $j=0,\ldots, N^q_{k}-1$,
\begin{equation}
D_{k,N^q_{k}-j}=\frac{1}{j!}\left[\frac{\partial^j}{\partial s^j}\left(\psi^-(s)(s+\gamma_{k,q})^{N^q_{k}}\right)\right]_
{s=-\gamma_{k,q}}.
\end{equation}
\end{Lemma}

\begin{Remark}
In what follows, $\psi^+(s)$ in (2.10) and $\psi^-(s)$ in (2.15) are treated as rational functions, i.e.,we extend the definition of $\psi^+(s)$$(\psi^-(s))$ to the whole complex plane except at $-\beta_{1,q}, \ldots, -\beta_{M_q^+,q}$$(-\gamma_{1,q}, \ldots, -\gamma_{N_q^-,q})$.
\end{Remark}

The following lemma can be obtained easily by using Lemmas 2.2 and 2.3. But for the convenience of the reader, we give the details.
\begin{Lemma}
(1) For $q>0$ and $x, y \leq 0$, we have
\begin{equation}
\begin{split}
\mathbb E\left[e^{-q\tau_{x,X}^-}\textbf{1}_{\{X_{\tau_{x,X}^-}-x \in dy\}}\right]=D_0(x)\delta_0(dy)+
\sum_{k=1}^{n^-}\sum_{j=1}^{n_k}D_{kj}(x)\frac{(\vartheta_k)^j(-y)^{j-1}}{(j-1)!}e^{\vartheta_k y}dy,
\end{split}
\end{equation}
where $\delta_0(dy)$ is the Dirac delta at $y=0$; for fixed $x \leq 0$, $D_0(x)$ and $D_{kj}(x)$ are given by rational expansion:
\begin{equation}
\begin{split}
D_0(x)+\sum_{k=1}^{n^-}\sum_{j=1}^{n_k}D_{kj}(x)\left(\frac{\vartheta_k}{\vartheta_k+ s}\right)^j
&=\frac{D_1e^{\gamma_{1,q}x}}{\psi^-(s)(s+\gamma_{1,q})}
+\frac{1}{\psi^-(s)}\sum_{k=2}^{N_q^-}\sum_{j=1}^{N_k^q}D_{kj} \int_{-\infty}^{0}\frac{(-y-x)^{j-1}}{(j-1)!}e^{sy+\gamma_{k,q}(y+x)}dy\\
&=\frac{D_1 e^{\gamma_{1,q}x}}{\psi^-(s)(s+\gamma_{1,q})}
+\frac{1}{\psi^-(s)}\sum_{k=2}^{N_q^-}\sum_{j=1}^{N^q_{k}}D_{kj}
\frac{(-1)^{j-1}}{(j-1)!}\frac{\partial^{j-1}}{\partial \gamma^{j-1}}\left(\frac{e^{\gamma x}}{s+\gamma}\right)_{\gamma=\gamma_{k,q}}.
\end{split}
\end{equation}

(2) For $q>0$, $x>0$ and $y \geq  0$,
\begin{equation}
\tilde{\mathbb E}\left[e^{-q \tau_{x,Y}^+}\textbf{1}_{\{Y_{\tau_{x,Y}^+}-x \in dy\}}\right]=
C_0(x)\delta_0(dy)+
\sum_{k=1}^{m^+}\sum_{j=1}^{m_k}C_{kj}(x)\frac{(\eta_k)^jy^{j-1}}{(j-1)!}e^{-\eta_k y}dy,
\end{equation}
where for fixed $x\geq 0$, $C_0(x)$ and $C_{kj}(x)$ are given by rational expansion:
\begin{equation}
\begin{split}
C_0(x)+\sum_{k=1}^{m^+}\sum_{j=1}^{m_k}C_{kj}(x)\left(\frac{\eta_k}{\eta_k+ s}\right)^j
&=\frac{C_1 e^{-\beta_{1,q}x}}{\psi^+(s)(s+\beta_{1,q})}
+\frac{1}{\psi^+(s)}\sum_{k=2}^{M_q^+}\sum_{j=1}^{M^q_{k}}C_{kj}
\int_{0}^{\infty}\frac{(y+x)^{j-1}}{(j-1)!}e^{-sy-\beta_{k,q}(y+x)}dy\\
&=\frac{C_1 e^{-\beta_{1,q}x}}{\psi^+(s)(s+\beta_{1,q})}
+\frac{1}{\psi^+(s)}\sum_{k=2}^{M_q^+}\sum_{j=1}^{M^q_{k}}C_{kj}
\frac{(-1)^{j-1}}{(j-1)!}\frac{\partial^{j-1}}{\partial \beta^{j-1}}\left(\frac{e^{-\beta x}}{s+\beta}\right)_{\beta=\beta_{k,q}}.
\end{split}
\end{equation}
\end{Lemma}

\begin{proof}
First, it holds that (see Corollary 2 in Alili and Kyprianou (2005))
\begin{equation}
\mathbb E\left[e^{-q \tau_{x,X}^-+sX_{\tau_{x,X}^-}}\right]=\frac{\mathbb E\left[
e^{s\underline{X}_{e(q)}}\textbf{1}_{\{\underline{X}_{e(q)}< x \}}\right]}{\mathbb E\left[
e^{s\underline{X}_{e(q)}}\right]}, \ \ s \geq 0.
\end{equation}
Note that
\[
\int_{-\infty}^{0}(-y-x)^{j-1}e^{sy+\gamma_{k,q}(y+x)}dy=(-1)^{j-1}\frac{\partial^{j-1}}{\partial \gamma^{j-1}}\left(\frac{e^{\gamma x}}{s+\gamma}\right)_{\gamma=\gamma_{k,q}}.
\]
For $s>0$, it follows from (2.15), (2.16) and (2.23) that
\begin{equation}
\begin{split}
&\mathbb E\left[e^{-q \tau_{x,X}^-+s(X_{\tau_{x,X}^-}-x)}\right]=
\prod_{k=1}^{n^-}\left(\frac{\vartheta_k}{s+\vartheta_k}\right)^{n_k}
\prod_{k=1}^{N_q^-}\left(\frac{s+\gamma_{k,q}}{\gamma_{k,q}}\right)^{N^q_{k}}
\frac{D_1 e^{\gamma_{1,q}x}}{(s+\gamma_{1,q})}\\
&+\prod_{k=1}^{n^-}\left(\frac{\vartheta_k}{s+\vartheta_k}\right)^{n_k}
\prod_{k=1}^{N_q^-}\left(\frac{s+\gamma_{k,q}}{\gamma_{k,q}}\right)^{N^q_{k}}\sum_{k=2}^{N_q^-}\sum_{j=1}^{N^q_{k}}D_{kj}
\frac{(-1)^{j-1}}{(j-1)!}\frac{\partial^{j-1}}{\partial \gamma^{j-1}}\left(\frac{e^{\gamma x}}{s+\gamma}\right)_{\gamma=\gamma_{k,q}}.
\end{split}
\end{equation}

Note that formula (2.14) holds and the right-hand side of (2.24) is a rational function of $s$. For fixed $x \leq 0$, by rational expansion, we can obtain that there are some constants $D_0(x)$ and $D_{kj}(x)$ (depend only on $x$) such that
\[
\mathbb E\left[e^{-q \tau_{x,X}^-+s(X_{\tau_{x,X}^-}-x)}\right]=
D_0(x)+\sum_{k=1}^{n^-}\sum_{j=1}^{n_k}D_{kj}(x)\left(\frac{\vartheta_k}{\vartheta_k+ s}\right)^j,
\]
from which (2.19) is derived.

Similarly, for the proof of (2.21), we first use (2.10), (2.11) and the known result (see formula (4) in Alili and Kyprianou (2005))
\begin{equation}
\tilde{\mathbb E}\left[e^{-q \tilde{\tau}_{x,Y}^+-sY_{\tilde{\tau}_{x,Y}^+}}\right]=\frac{\tilde{\mathbb E}\left[
e^{-s\overline{Y}_{e(q)}}\textbf{1}_{\{\overline{Y}_{e(q)} > x \}}\right]}{\tilde{\mathbb E}\left[
e^{-s\overline{Y}_{e(q)}}\right]}, \ \ for \ \ x, s \geq 0,
\end{equation}
where $\tilde{\tau}_{x,Y}^+:=\inf\{t \geq 0: Y_t > x\}$, and then note that $\tilde{\mathbb P}\left(\tau_{x,Y}^+=\tilde{\tau}_{x,Y}^+\right)=1$ for $x>0$ since $\sigma > 0$.
\end{proof}

\begin{Remark}
Formulas (2.20) and (2.22) give
\[
D_0(0)+\sum_{k=1}^{n^-}\sum_{j=1}^{n_k}D_{kj}(0)\left(\frac{\vartheta_k}{\vartheta_k+ s}\right)^j=1,\ \
C_0(0)+\sum_{k=1}^{m^+}\sum_{j=1}^{m_k}C_{kj}(0)\left(\frac{\eta_k}{\eta_k+ s}\right)^j=1,
\]
which means
\begin{equation}
\begin{split}
&C_0(0)=D_0(0)=1 \ \ and  \ \  C_{kj}(0)=D_{kj}(0)=0.
\end{split}
\end{equation}
In addition, for fixed $x \leq 0$, formula (2.20) implies that $D_0(x)$ and $D_{kj}(x)$ have the following forms:
\[
\sum_{k=1}^{N_q^-}\sum_{j=1}^{N^q_{k}}\hat{D}_{kj}\frac{(-x)^{j-1}}{(j-1)!}
e^{\gamma_{k,q}x},
\]
where $\hat{D}_{kj}$ is a constant and not dependent on $x$. For $C_0(x)$ and $C_{kj}(x)$, similar results can be drawn from (2.22).
\end{Remark}

Some straightforward calculations give us the following lemma.

\begin{Lemma} For any $\theta >0$ and $s \geq 0$ with $\theta \neq s$,
\begin{equation}
\begin{split}
&\int_{0}^{\infty}e^{-\theta x}C_0(x)dx+\sum_{k=1}^{m^+}\sum_{j=1}^{m_k} \int_{0}^{\infty}e^{-\theta x}C_{kj}(x)dx\left(\frac{\eta_k}{\eta_k+s}\right)^j=\frac{1}{s-\theta}\left(\frac{\psi^+(\theta)}
{\psi^+(s)}-1\right),
\end{split}
\end{equation}
and
\begin{equation}
\begin{split}
&\int_{-\infty}^{0}e^{\theta x} D_0(x)dx +\sum_{k=1}^{n^-}\sum_{j=1}^{n_k} \int_{-\infty}^{0}e^{\theta x}D_{kj}(x)dx\left(\frac{\vartheta_k}{\vartheta_k+s}\right)^j=\frac{1}{s-\theta}\left(\frac{\psi^-(\theta)}
{\psi^-(s)}-1\right),
\end{split}
\end{equation}
where the functions $\psi^+(\cdot)$ and $\psi^-(\cdot)$ are given by (2.10) and (2.15), respectively.
\end{Lemma}

\begin{proof}
Note first that
\[
\int_{0}^{\infty}e^{-\theta x}\frac{\partial^{j-1}}{\partial \beta^{j-1}}\left(\frac{e^{-\beta x}}{s+\beta}\right)_{\beta=\beta_{k,q}}dx=\frac{\partial^{j-1}}{\partial \beta^{j-1}}\left(
\frac{1}{s-\theta}\left(\frac{1}{\theta+\beta}-\frac{1}{s+\beta}\right)\right)_{\beta=\beta_{k,q}}.
\]
Applying (2.10) and (2.22) leads to (2.27). The proof of (2.28) is similar.
\end{proof}

\begin{Remark}
In fact, identity (2.27) holds for a general  L\'evy process, i.e.,
\begin{equation}
\int_{0}^{\infty}e^{-\theta x} \tilde{\mathbb E}\left[e^{-q\tilde{\tau}_{x,Y}^+-s(Y_{\tilde{\tau}_{x,Y}^+}-x)}\right]dx=\frac{1}{s-\theta}
\left(\frac{\tilde{\mathbb E}\left[e^{-\theta \bar{Y}_{e(q)}}\right]}
{\tilde{\mathbb E}\left[e^{-s \bar{Y}_{e(q)}}\right]}-1\right),
\end{equation}
where $Y_t$ can be assumed to be a general L\'evy process and $\tilde{\tau}_{x,Y}^+:=\inf\{t \geq 0: Y_t > x\}$.  Formula (2.29) is known as the  Pecherskii Rogozin identity, which can be found from section 3.1 in Alili and Kyprianou (2005) or formula (3.2) in Pecherskii and Rogozin (1969).
\end{Remark}

\begin{Remark}
Compared with our previous paper (Zhou and Wu (2015)), the most different part in this one is that it is very difficult to deal with the expressions of $C_0(x)$, $C_{kj}(x)$, $D_0(x)$ and $D_{kj}(x)$ in Lemma 2.4. For example, we cannot find an easy approach to compute $\int_{0}^{\infty}e^{-\theta x}C_{kj}(x)dx$. In Section 3 below, we modify the method in Zhou and Wu (2015) and find that results in Lemma 2.5 are enough to derive the final outcomes.
\end{Remark}

\section{Main results}

In this section, we want to derive the expression of the expectation of $\int_{0}^{e(q)}\textbf{1}_{\{U_s < b\}}ds$ and we first deduce its  Laplace transform.

\begin{Theorem}
For any given $p, q > 0$, we have
\begin{equation}
\mathbb E_x\left[e^{-p \int_{0}^{e(q)}\textbf{1}_{\{U_t<b\}}dt}\right]=\left\{\begin{array}{cc}
\frac{q}{\xi}+\sum_{m=1}^{M_{\xi}^+}\sum_{i=1}^{M^{\xi}_{m}}H_{mi}\frac{(b-x)^{i-1}}{(i-1)!}e^{\beta_{m,\xi}(x-b)}, & x \leq b,\\
1+\sum_{n=1}^{N_q^-}\sum_{i=1}^{N^q_{n}}G_{ni}\frac{(x-b)^{i-1}}{(i-1)!}e^{\gamma_{n,q}(b-x)},&  x \geq b,
\end{array}\right.
\end{equation}
where $\xi=p+q$, for $1\leq m \leq M_{\xi}^+$, $0 \leq i \leq M^{\xi}_{m}-1$,
\begin{equation}
(-1)^{M^{\xi}_{m}-i}H_{m,M^{\xi}_{m}-i}=\frac{\partial^{i}}{i!\partial x^i}\left(f(x)(x-\beta_{m,\xi})^{M^{\xi}_{m}}\right)_{x=\beta_{m,\xi}},
\end{equation}
and for $1\leq n \leq N_q^-$, $0\leq i \leq N^q_{n}-1$,
\begin{equation}
G_{n,N^q_{n}-i}=\frac{\partial^{i}}{i!\partial x^i}\left(f(x)(x+\gamma_{n,q})^{N^q_{n}}\right)_{x=-\gamma_{n,q}}.
\end{equation}
Here, in (3.2) and (3.3), the function $f(x)$ is given by
\begin{equation}
f(x)=\frac{p}{\xi}\frac{\prod_{m=1}^{M_{\xi}^+}(-\beta_{m,\xi})^{M^{\xi}_{m}}
\prod_{n=1}^{N_q^-}(\gamma_{n,q})^{N^q_{n}}\prod_{k=1}^{m^+}(x-\eta_k)^{m_k}
\prod_{k=1}^{n^-}(x+\vartheta_k)^{n_k}}{ \prod_{k=1}^{m^+}(-\eta_k)^{m_k}\prod_{k=1}^{n^-}(\vartheta_k)^{n_k}x \prod_{m=1}^{M_{\xi}^+}(x-\beta_{m,\xi})^{M^{\xi}_{m}}\prod_{n=1}^{N_q^-}(x+\gamma_{n,q})^{N^q_{n}}}.
\end{equation}
\end{Theorem}

\begin{proof}
The derivation consists of three steps. The first step is similar to step (1) in the proof of Theorem 3.1 in Zhou and Wu (2015).

Step 1. Define a function of $x$ as
\begin{equation}
V(x) = \mathbb E_x\left[e^{-p\int_{0}^{e(q)}\textbf{1}_{\{U_s < b\}}ds}\right].
\end{equation}

For $x < b$, it follows from the strong Markov property that
\begin{equation}
\begin{split}
&V(x)
=\mathbb E_x\left[\int_{0}^{\kappa_b^+}q e^{-q t}e^{-p t}dt\right] + \mathbb E_x\left[\int_{\kappa_b^+}^{+\infty}q e^{-q t}e^{-p \int_{0}^{t}\textbf{1}_{\{U_s < b\}}ds}dt\right]\\
&=\frac{q}{\xi}\left(1 - \mathbb E_x\left[e^{-\xi \kappa_b^+}\right]\right)+\mathbb E_x\left[e^{-\xi\kappa_b^+}V(U_{\kappa_b^+})\right]
=\frac{q}{\xi}\Big(1 - \tilde{\mathbb E}_x\left[e^{-\xi \tau_{b,Y}^+}\right]\Big)+ \tilde{\mathbb E}_x\left[e^{-\xi \tau_{b,Y}^+}V(Y_{\tau_{b,Y}^+})\right],
\end{split}
\end{equation}
where $\xi =p+q$ and in the third equality we have used that $\{Y_t, t < \tau_{b,Y}^+ \}$ under $\tilde{\mathbb P}_x$  and $\{U_t, t < \kappa_b^+ \}$ under $\mathbb P_x$ have the same law when $x < b$.
Applying (2.21) to (3.6), we will obtain that
\begin{equation}
\begin{split}
V(x)
&=\sum_{k=1}^{m^+}\sum_{j=1}^{m_k}C_{kj}(b-x) \left(\int_{0}^{\infty}\frac{(\eta_k)^jy^{j-1}}{(j-1)!}e^{-\eta_k y}V(b+y)dy-\frac{q}{\xi}\right)
+\frac{q}{\xi}+C_0(b-x)\left(V(b)-\frac{q}{\xi}\right), \ \ x < b.
\end{split}
\end{equation}
Note that $C_0(b-x)$ and $C_{kj}(b-x)$ in (3.7) depend on $\xi$.

For $x \geq b$, using the strong Markov property again, we can obtain that
\begin{equation}
\begin{split}
V(x)
&=\mathbb E_x\left[ \int_{\kappa_b^-}^{+\infty}q e^{-q t}e^{-p\int_{\kappa_b^-}^{t}\textbf{1}_{\{U_s < b\}}ds}dt\right]+
\mathbb E_x\left[1-e^{-q\kappa_b^-}\right]\\
&=\mathbb E_x\left[e^{-q\kappa_b^-}V(U_{\kappa_b^-})\right]+\mathbb E_x\left[1-e^{-q \kappa_b^-}\right] =\mathbb E_x\left[e^{-q \tau_{b,X}^-}V(X_{\tau_{b,X}^-})\right]+\mathbb E_x\left[1-e^{-q \tau_{b,X}^-}\right]\\
&=\sum_{k=1}^{n^-}\sum_{j=1}^{n_k}D_{kj}(b-x)\left(\int_{-\infty}^{0}(\vartheta_k)^j\frac{(-y)^{j-1}}{(j-1)!}
e^{\vartheta_k y}V(b+y)dy-1\right)
 +1+D_0(b-x)\left(V(b)-1\right),
\end{split}
\end{equation}
where the third equality follows from the fact that $\{X_t, t < \tau_{b,X}^-\}$ and $\{U_t, t < \kappa_b^-\}$ under $\mathbb P_x$ have the same law for $x \geq b$; and we have used (2.19) in the fourth equality.

From Remark 2.3, (3.7) and (3.8), we can deduce that $V(x)$ must have the following form:
\begin{equation}
V(x)=\left\{\begin{array}{cc}
\frac{q}{\xi}+\sum_{k=1}^{M_{\xi}^+}\sum_{j=1}^{M^{\xi}_{k}}H_{kj}\frac{(b-x)^{j-1}}{(j-1)!}
e^{-\beta_{k,\xi}(b-x)}, & x < b,\\
1+\sum_{k=1}^{N_q^-}\sum_{j=1}^{N^q_{k}}G_{kj}\frac{(x-b)^{j-1}}{(j-1)!}
e^{\gamma_{k,q}(b-x)}, & x \geq b,
\end{array}\right.
\end{equation}
where $H_{kj}$ and $G_{kj}$ do not depend on $x$. Thus, the remaining thing is to derive expressions for $H_{kj}$ and $G_{kj}$ in (3.9), which will be done in the next two steps.

Step 2. It is easy to see from (3.7), (3.8) and (2.26) that $V(x)$ is continuous at $b$, thus
\begin{equation}
\sum_{n=1}^{N_q^-}G_{n1}+\frac{p}{\xi}=\sum_{m=1}^{M_{\xi}^+}H_{m1}.
\end{equation}
In addition, the derivative of $V(x)$ at $b$ is also continuous (see Remark A.5 in Zhou and Wu (2015)). This will lead to
\begin{equation}
-\sum_{n=1}^{N_q^-}G_{n1}\gamma_{n,q}+\sum_{n=1}^{N_q^-}G_{n2}\textbf{1}_{\{N^q_{n} \geq 2\}}=\sum_{m=1}^{M_{\xi}^+}H_{m1}\beta_{m,\xi}-\sum_{m=1}^{M_{\xi}^+}H_{m2}\textbf{1}_{\{M^{\xi}_{m}\geq 2\}}.
\end{equation}
Applying the expression of $V(x)$ for $x \geq b$ in (3.9), we can derive the following result from (3.7):
\begin{equation}
\begin{split}
V(x)
&=\sum_{k=1}^{m^+}\sum_{j=1}^{m_k}C_{kj}(b-x)\left\{\sum_{n=1}^{N_q^-}\sum_{i=1}^{N^q_{n}}
\frac{G_{ni}(\eta_k)^j(i+j-2)!}{(\eta_k+\gamma_{n,q})^{i+j-1}(i-1)!(j-1)!}+\frac{p}{\xi}\right\}
 +\frac{q}{\xi}+C_0(b-x)\left(V(b)-\frac{q}{\xi}\right), \ \ for \ \ x< b.
\end{split}
\end{equation}
Therefore, it follows from (3.9), (3.12) and the fact that $V(b)=1+\sum_{n=1}^{N_q^-}G_{n1}$ (which is due to (3.9)), we obtain that the following identity must hold for all $x < b$:
\begin{equation}
\begin{split}
&\sum_{k=1}^{m^+}\sum_{j=1}^{m_k}C_{kj}(b-x)\left\{\sum_{n=1}^{N_q^-}\sum_{i=1}^{N^q_{n}}
\frac{G_{ni}(\eta_k)^j(i+j-2)!}{(\eta_k+\gamma_{n,q})^{i+j-1}(i-1)!(j-1)!}+\frac{p}{\xi}\right\}
+C_0(b-x)\left(\sum_{n=1}^{N_q^-}G_{n1}+\frac{p}{\xi}\right)\\
&=\sum_{m=1}^{M_{\xi}^+}\sum_{i=1}^{M^{\xi}_{m}}
H_{mi}\frac{(b-x)^{i-1}}{(i-1)!}e^{-\beta_{m,\xi}(b-x)}.
\end{split}
\end{equation}
Note first that
\[
\frac{G_{ni}(\eta_k)^j(i+j-2)!}{(\eta_k+\gamma_{n,q})^{i+j-1}(i-1)!(j-1)!}=\frac{G_{ni}(-1)^{i-1}}{(i-1)!}
\frac{\partial^{i-1}}{\partial s^{i-1}}\left(\frac{(\eta_k)^j}{(\eta_k+s)^j}\right)_{s=\gamma_{n,q}}.
\]
Next, for given $\theta > 0$, multiplying both sides of (3.13) by $e^{\theta (x-b)}$ and integrating with respect to $x$ from $-\infty$ to $b$ and interchanging the order of sum will yield
\begin{equation}
\begin{split}
& \ \ \sum_{n=1}^{N_q^-}\sum_{i=1}^{N^q_{n}}\frac{G_{ni}(-1)^{i-1}}{(i-1)!}\frac{\partial^{i-1}}{\partial s^{i-1}}\left(\frac{\psi^+(\theta)}{\psi^+(s)(s-\theta)}-\frac{1}{s-\theta}\right)_{s=\gamma_{n,q}}
+\frac{p(1-\psi^+(\theta))}{\xi \theta}\\
&=\sum_{m=1}^{M_{\xi}^+}\sum_{i=1}^{M^{\xi}_{m}}\frac{H_{mi}}{(\theta+\beta_{m,\xi})^i},
\end{split}
\end{equation}
where in the derivation of the left-hand side of (3.14), we have used (2.27).

As both sides of (3.14) are rational functions of $\theta$, we can extend identity (3.14) to the whole complex plane except at $0$, $-\beta_{1,\xi},\ldots, -\beta_{M_{\xi}^+,\xi}$ and $\gamma_{1,q}, \ldots, \gamma_{N_q^-,q}$. Then, for any given $1 \leq k \leq m^+$ and $0\leq j \leq m_k-1$, taking a derivative on both sides of (3.14) with respect to $\theta $ up to $j$ order and letting $\theta$ equal to $-\eta_k$, we have
\begin{equation}
\begin{split}
&\sum_{n=1}^{N_q^-}\sum_{i=1}^{N^q_{n}}\frac{G_{ni}(-1)(i+j-1)!}{(i-1)!(\eta_k+\gamma_{n,q})^{i+j}}
-\frac{p j!}{\xi(\eta_k)^{j+1}}=\sum_{m=1}^{M_{\xi}^+}\sum_{i=1}^{M^{\xi}_{m}}\frac{H_{mi}(i+j-1)!(-1)^i}{(i-1)!
(\eta_k-\beta_{m,\xi})^{i+j}},
\end{split}
\end{equation}
which can be proved by noting that $\frac{\partial^j}{\partial \theta^j}\big(\psi^+(\theta)\big)_{\theta=-\eta_k}=0$ for $j < m_k$ (see (2.10)).

Similar to the derivation of (3.13), it follows from (3.8) and (3.9) that the following identity holds for all $x>b$:
\begin{equation}
\begin{split}
&\sum_{k=1}^{n^-}\sum_{j=1}^{n_k}D_{kj}(b-x)\left\{\sum_{m=1}^{M_{\xi}^+}\sum_{i=1}^{M^{\xi}_{m}}
\frac{H_{mi}(\vartheta_k)^j(i+j-2)!}{(\vartheta_k+\beta_{m,\xi})^{i+j-1}(i-1)!(j-1)!}-\frac{p}{\xi}\right\} +D_0(b-x)\left(V(b)-1\right)\\
&=\sum_{n=1}^{N_q^-}
\sum_{i=1}^{N^q_{n}}G_{ni}\frac{(x-b)^{i-1}}{(i-1)!}e^{\gamma_{n,q}(b-x)}.
\end{split}
\end{equation}
From (2.28), (3.16) and the fact of $V(b)-1=\sum_{n=1}^{N_q^-}G_{n1}=\sum_{m=1}^{M_{\xi}^+}H_{m1}-\frac{p}{\xi}$ (see (3.10)), for $\theta >0$, we have (see the derivation of (3.14))
\begin{equation}
\begin{split}
&\sum_{m=1}^{M_{\xi}^+}\sum_{i=1}^{M^{\xi}_{m}}
\frac{H_{mi}(-1)^{i-1}}{(i-1)!}\frac{\partial^{i-1}}{\partial s^{i-1}}\left(
\frac{\psi^-(\theta)}{\psi^-(s)(s-\theta)}-\frac{1}{s-\theta}\right)_{s=\beta_{m,\xi}} -\frac{p(1-\psi^-(\theta))}{\xi \theta}\\
&=
\sum_{n=1}^{N_q^-}\sum_{i=1}^{N^q_{n}}\frac{G_{ni}}{(\theta+\gamma_{n,q})^i}.
\end{split}
\end{equation}
After that, similar to deriving (3.15),  for any given $1 \leq k \leq n^-$, $0\leq j \leq n_k-1$, if we take a derivative on both sides of (3.17) with respect to $\theta $ up to $j$ order and let $\theta$ equal to $-\vartheta_k$, then we can derive that (note that $\frac{\partial^j}{\partial \theta^j}\big(\psi^-(\theta)\big)_{\theta=-\vartheta_k}=0$, see (2.15))
\begin{equation}
\begin{split}
&\sum_{m=1}^{M_{\xi}^+}\sum_{i=1}^{M^{\xi}_{m}}
\frac{H_{mi}(-1)(i+j-1)!}{(i-1)!(\vartheta_k+\beta_{m,\xi})^{i+j}}+\frac{p j!}{\xi (\vartheta_k)^{j+1}}
=
\sum_{n=1}^{N_q^-}\sum_{i=1}^{N^q_{n}}\frac{G_{ni}(-1)^i(i+j-1)!}{(i-1)!(\vartheta_k-\gamma_{n,q})^{i+j}}.
\end{split}
\end{equation}

Until now, we have derived some equations, from which expressions for $H_{kj}$ and $G_{kj}$ in (3.9) can be obtained. In the third step, we will give the details of solving these equations.

Step 3. Suppose $H_{mi}$ and $G_{ni}$ are solutions of (3.10), (3.11), (3.15) and (3.18), and define a function of $x$ as
\begin{equation}
f_1(x)=\sum_{m=1}^{M_{\xi}^+}\sum_{i=1}^{M^{\xi}_{m}}\frac{H_{mi}(-1)^i}{(x-\beta_{m,\xi})^i}+
\sum_{n=1}^{N_q^-}\sum_{i=1}^{N^q_{n}}\frac{G_{ni}}{(x+\gamma_{n,q})^i}+\frac{p}{\xi x}.
\end{equation}

For fixed $1 \leq k \leq m^+$ and $0\leq j \leq m_k-1$, formula (3.15) yields that $\frac{\partial^j}{\partial x^j}\left(f_1(x)\right)_{x=\eta_k}=0$, which means that $\eta_k$ is a root of the equation $f_1(x)=0$ with multiplicity $m_k$. Similarly, from (3.18), for $1 \leq k \leq n^-$, we obtain that $-\vartheta_k$ is a $n_k$-multiplicity root of $f_1(x)=0$. So $f_1(x)$ can be written as (because $\sum_{m=1}^{M^+_{\xi}}M_m^{\xi}+\sum_{n=1}^{N_q^-}N_n^q=2+\sum_{k=1}^{m^+}m_k+\sum_{k=1}^{n^-}n_k$, see (2.9) and (2.14))
\begin{equation}
f_1(x)=\frac{\prod_{k=1}^{m^+}(x-\eta_k)^{m_k}\prod_{k}^{n^-}(x+\vartheta_k)^{n_k}(A_1x^2+A_2x+A_3)}
{x\prod_{m=1}^{M_{\xi}^+}(x-\beta_{m,\xi})^{M^{\xi}_{k}}\prod_{n=1}^{N_q^-}(x+\gamma_{n,q})^{N^q_{n}}},
\end{equation}
with some proper constants $A_1,A_2, A_3$.

Formula (3.10) implies that $A_1=0$. For the value of $A_2$, we have
\begin{equation}
\begin{split}
A_2
&=\sum_{m=1}^{M_{\xi}^+}H_{m1}(-1)(A_4+\beta_{m,\xi})+\sum_{m=1}^{M_{\xi}^+}H_{m2}
\textbf{1}_{\{M^{\xi}_m\geq 2\}} +
\sum_{n=1}^{N_q^-}G_{n1}(A_4-\gamma_{n,q})+\sum_{n=1}^{N_q^-}G_{n2}\textbf{1}_{\{N^q_{n}\geq 2\}}+\frac{p}{\xi}A_4\\
&=0,
\end{split}
\end{equation}
where $A_4=-\sum_{m=1}^{M_{\xi}^+}M_m^{\xi}\beta_{m,\xi}+\sum_{n=1}^{N_q^-}N_n^q\gamma_{n,q}$ and the second equality is due to (3.10) and (3.11).
Furthermore, note that $\lim_{x\rightarrow 0}f_1(x)x=\frac{p}{\xi}$ in (3.19). Hence,
\begin{equation}
A_3=\frac{p\prod_{m=1}^{M_{\xi}^+}(-\beta_{m,\xi})^{M^{\xi}_{m}}
\prod_{n=1}^{N_q^-}(\gamma_{n,q})^{N^q_{n}}}{\xi \prod_{k=1}^{m^+}(-\eta_k)^{m_k}\prod_{k=1}^{n^-}(\vartheta_k)^{n_k}}.
\end{equation}
Therefore, $f_1(x)$ equals $f(x)$ given by (3.4). Thus, formulas (3.2) and (3.3) are derived from (3.4) and (3.19) by writing $f(x)$ in its fraction expansion.
\end{proof}

From Theorem 3.1, the expectation of $\int_{0}^{e(q)}\textbf{1}_{\{U_s < b\}}ds$ is deduced.
\begin{Corollary}
For any given $q > 0$,
\begin{equation}
\begin{split}
&\mathbb E_x\left[\int_{0}^{e(q)}\textbf{1}_{\{U_t<b\}}dt\right]
=\left\{\begin{array}{cc}
\frac{1}{q}-\frac{1}{q}\sum_{m=1}^{M_q^+}\sum_{i=1}^{M^q_{m}}H^0_{mi}\frac{(b-x)^{i-1}}{(i-1)!}e^{\beta_{m,q}(x-b)}, & x \leq b,\\
-\frac{1}{q}\sum_{n=1}^{N_q^-}\sum_{i=1}^{N^q_{n}}G^0_{ni}\frac{(x-b)^{i-1}}{(i-1)!}e^{\gamma_{n,q}(b-x)},&  x \geq b,
\end{array}\right.
\end{split}
\end{equation}
where $H^0_{mi}$ and $G^0_{ni}$ are given by rational expansion:
\begin{equation}
\begin{split}
&\sum_{m=1}^{M_{q}^+}\sum_{i=1}^{M^{q}_{m}}\frac{H^0_{mi}(-1)^i}{(x-\beta_{m,q})^i}+
\sum_{n=1}^{N_q^-}\sum_{i=1}^{N^q_{n}}\frac{G^0_{ni}}{(x+\gamma_{n,q})^i}+\frac{1}{x}
=f^0(x),
\end{split}
\end{equation}
i.e.,
\begin{equation}
\begin{split}
&(-1)^{M^q_{m}-i}H^0_{m,M^q_{m}-i}=\frac{\partial^{i}}{i!\partial x^i}\left(f^0(x)(x-\beta_{m,q})^{M^q_{m}}\right)_{x=\beta_{m,q}},\\
&G^0_{n,N^q_{n}-i}=\frac{\partial^{i}}{i!\partial x^i}\left(f^0(x)(x+\gamma_{n,q})^{N^q_{n}}\right)_{x=-\gamma_{n,q}}.
\end{split}
\end{equation}
Here $f^0(x)$ is a rational function and is given by
\begin{equation}
f^0(x)=\frac{\prod_{m=1}^{M_q^+}(-\beta_{m,q})^{M^q_{m}}
\prod_{n=1}^{N_q^-}(\gamma_{n,q})^{N^q_{n}}\prod_{k=1}^{m^+}(x-\eta_k)^{m_k}
\prod_{k=1}^{n^-}(x+\vartheta_k)^{n_k}}{ \prod_{k=1}^{m^+}(-\eta_k)^{m_k}\prod_{k=1}^{n^-}(\vartheta_k)^{n_k}x \prod_{m=1}^{M_q^+}(x-\beta_{m,q})^{M^q_{m}}\prod_{n=1}^{N_q^-}(x+\gamma_{n,q})^{N^q_{n}}}.
\end{equation}
\end{Corollary}

\begin{proof}
Since $p>0$, it is clear that $0\leq \mathbb E_x\left[e^{-p\int_{0}^{e(q)}\textbf{1}_{\{U_s < b\}}ds}\right] \leq 1$. Recalling Lemma 2.3 and using (3.1), we have
\[
\lim_{x \uparrow \infty}\frac{\mathbb E_x\left[e^{-p\int_{0}^{e(q)}\textbf{1}_{\{U_s < b\}}ds}\right]-1}
{e^{\gamma_{1,q}(b-x)}}=G_{11}.
\]
Note that the expectation $\mathbb E\left[e^{\phi\underline{X}_{e(q)}}\right]$ exists for $Re(\phi)>-\gamma_{1,q}$ (see (2.15)).

From (3.1), (3.19) and the fact that $f_1(x)=f(x)$, for given $p, q>0$ and $-\gamma_{1,q} < Re(\phi) <0$, we know
\begin{equation}
\begin{split}
&\int_{-\infty}^{\infty}e^{-\phi(x-b)}\left\{\mathbb E_x\left[e^{-p\int_{0}^{e(q)}\textbf{1}_{\{U_s < b\}}ds}\right]-1\right\}dx=f(\phi)=\frac{p}{\phi(p+q)}\tilde{\mathbb E}\left[e^{\phi\overline{Y}_{e(p+q)}}\right]
\mathbb E\left[e^{\phi\underline{X}_{e(q)}}\right],
\end{split}
\end{equation}
where $f(\cdot)$ is given by (3.4) and the second equality follows from (2.10), (2.15) and (3.4). We can prove that (see Lemma 3.1 below)
\begin{equation}
\begin{split}
&\int_{-\infty}^{\infty}e^{-\phi(x-b)}\mathbb E_x\left[\int_{0}^{e(q)}\textbf{1}_{\{U_s < b\}}ds\right]dx
=-\frac{1}{q \phi}\tilde{\mathbb E}\left[e^{\phi\overline{Y}_{e(q)}}\right]
\mathbb E\left[e^{\phi\underline{X}_{e(q)}}\right],
\end{split}
\end{equation}
which holds for $-\gamma<Re(\phi)<0$, where $\gamma$ is a constant and satisfies $0<\gamma\leq\gamma_{1,q}$.

Define
\[
g(x)=\left\{\begin{array}{cc}
\frac{1}{q}-\frac{1}{q}\sum_{m=1}^{M_q^+}\sum_{i=1}^{M^q_{m}}H^0_{mi}\frac{(b-x)^{i-1}}{(i-1)!}e^{\beta_{m,q}(x-b)}, & x < b,\\
-\frac{1}{q}\sum_{n=1}^{N_q^-}\sum_{i=1}^{N^q_{n}}G^0_{ni}\frac{(x-b)^{j-1}}{(j-1)!}e^{\gamma_{n,q}(b-x)},&  x \geq b,
\end{array}\right.
\]
where $H^0_{mi}$ and $G^0_{ni}$ are given by (3.25).
For $-\gamma_{1,q} < Re(\phi) < 0$, we can show that
\begin{equation}
\int_{-\infty}^{\infty}e^{-\phi (x-b)}g(x)dx=-\frac{1}{q \phi}\tilde{\mathbb E}\left[e^{\phi\overline{Y}_{e(q)}}\right]
\mathbb E\left[e^{\phi\underline{X}_{e(q)}}\right],
\end{equation}
where the equality follows from (3.24) and (3.26). Formulas (3.28) and (3.29) imply that
both sides of (3.23) have the same Laplace (or Fourier) transform.

Obviously, $g(x)$ is continuous on $\mathbb R/\{b\}$. As the coefficient of $x^{\sum_{m=1}^{M_q^+}M_m^q+\sum_{n=1}^{N_q^-}N_n^q}$ in the numerator of $f^0(x)$ in (3.26) is $0$ (note that $\sum_{m=1}^{M_q^+}M_m^q+\sum_{n=1}^{N_q^-}N_n^q=2+\sum_{m=1}^{m^+}m_k+\sum_{n=1}^{n^-}n_k$),
 it follows from (3.24) and (3.26) that
\begin{equation}
1-\sum_{m=1}^{M_q^+}H^0_{m1}+\sum_{n=1}^{N_q^-}G^0_{n1}=0.
\end{equation}
This gives that $g(x)$ is also continuous at $b$, thus  $g(x)$ is continuous on $\mathbb R$.

If we define $V_1(x)=\mathbb P_x\left(U_{e(q)}<b\right)$, then similar computations as in (3.6) and (3.8) will produce
\begin{equation}
V_1(x)=\left\{\begin{array}{cc}
\mathbb E_x\left[e^{-q\tau_{b,X}^-}V_1(X_{\tau_{b,X}^-})\right], \ \ x \geq b,\\
1+\tilde{\mathbb E}_x\left[e^{-q\tau_{b,Y}^+}\left(V_1(Y_{\tau_{b,Y}^+})-1\right)\right],&  x < b,
\end{array}\right.
\end{equation}
which combined with (2.19), (2.21) and Remark 2.3, leads to that $V_1(x)$ is continuous on $\mathbb R$.
This implies that $\mathbb E_x\left[\int_{0}^{e(q)}\textbf{1}_{\{U_s < b\}}ds\right]$, as a function of $x$, is also continuous on $\mathbb R$ (see (3.32) below).

Therefore, formula (3.23) is derived from (3.28) and (3.29).
\end{proof}

\begin{Lemma}
There is a constant $0<\gamma \leq \gamma_{1,q}$ such that (3.28) holds for $-\gamma<Re(\phi)<0$.
\end{Lemma}
\begin{proof}
For $p>0$, it is obvious that
\[
\begin{split}
\frac{\partial}{\partial p}\left(1-\mathbb E_x\left[e^{-p\int_{0}^{e(q)}\textbf{1}_{\{U_s < b\}}ds}\right]\right)
&=\mathbb E_x\left[\int_{0}^{e(q)}\textbf{1}_{\{U_s < b\}}dse^{-p\int_{0}^{e(q)}\textbf{1}_{\{U_s < b\}}ds}\right]\leq \mathbb E_x\left[\int_{0}^{e(q)}\textbf{1}_{\{U_s < b\}}ds\right].
\end{split}
\]
Besides, applying integration by parts yields
\begin{equation}
q\mathbb E_x\left[\int_{0}^{e(q)}\textbf{1}_{\{U_t<b\}}dt\right]=q\int_{0}^{\infty}e^{-qt}\mathbb P_x(U_t<b)dt=\mathbb P_x\left(U_{e(q)}<b\right).
\end{equation}
Hence,
\[
q\mathbb E_x\left[\int_{0}^{e(q)}\textbf{1}_{\{U_s < b\}}ds\right]\leq
\left\{
\begin{array}{cc}
q\mathbb E_x\left[e(q)\right]=1,&x\leq b,\\
\mathbb P_x\left(U_{e(q)}<b\right) \leq
\tilde{\mathbb P}_x\left(\underline{Y}_{e(q)}<b\right) \vee
\mathbb P_x\left(\underline{X}_{e(q)}<b\right) ,&x>b.
\end{array}\right.
\]
It follows from (2.16) that $\mathbb P\left(\underline{X}_{e(q)} < x\right)\thicksim \frac{D_1}{\gamma_{1,q}}e^{\gamma_{1,q}x}$ when $x\downarrow -\infty$. Similarly, we can conclude that there are two constants $\hat{D}_1$ and $\hat{\gamma}_{1,q}>0$ such that $\tilde{\mathbb P}
\left(\underline{Y}_{e(q)} < x\right)\thicksim \frac{\hat{D}_1}{\hat{\gamma}_{1,q}}e^{\hat{\gamma}_{1,q}x}$
when $x\downarrow -\infty$.

For $\max(-\gamma_{1,q},-\hat{\gamma}_{1,q})<Re(\phi)<0$, the above results give
\[
\int_{-\infty}^{\infty}|e^{-\phi(x-b)}|\mathbb E_x\left[\int_{0}^{e(q)}\textbf{1}_{\{U_s < b\}}ds\right]dx < \infty,
\]
which leads to
\[
\frac{\partial }{\partial p}\int_{-\infty}^{\infty}e^{-\phi(x-b)}\left\{\mathbb E_x\left[e^{-p\int_{0}^{e(q)}\textbf{1}_{\{U_s < b\}}ds}\right]-1\right\}dx
=-\int_{-\infty}^{\infty}e^{-\phi(x-b)}\left\{\mathbb E_x\left[
\int_{0}^{e(q)}\textbf{1}_{\{U_s < b\}}dse^{-p\int_{0}^{e(q)}\textbf{1}_{\{U_s < b\}}ds}\right]\right\}dx.
\]

In addition, for $Re(\phi)<0$, from Theorem 6.16 in Kyprianou (2006), we have
\[
\tilde{\mathbb E}\left[e^{\phi\overline{Y}_{e(p+q)}}\right]=e^{\int_{0}^{\infty}\frac{1}{t}e^{-(p+q)t}
\int_{-\infty}^{0}(e^{-\phi x}-1)\tilde{\mathbb P}\left(Y_t \in dx\right)dt},
\]
thus
\begin{equation}
\frac{\partial}{\partial p}\tilde{\mathbb E}\left[e^{\phi\overline{Y}_{e(p+q)}}\right]
=-
\int_{0}^{\infty}e^{-(p+q)t}
\int_{-\infty}^{0}(e^{-\phi x}-1)\tilde{\mathbb P}\left(Y_t \in dx\right)dt
\times
e^{\int_{0}^{\infty}\frac{1}{t}e^{-(p+q)t}
\int_{-\infty}^{0}(e^{-\phi x}-1)\tilde{\mathbb P}\left(Y_t \in dx\right)dt}.
\end{equation}

Therefore, on both sides of (3.27), taking a derivative with respect to $p$ and then letting $p\downarrow 0$ will lead to (3.28), where $\gamma=\gamma_{1,q}\wedge \hat{\gamma}_{1,q}$.
\end{proof}

\begin{Remark}
Formula (3.32) means that the probability $\mathbb P_x\left(U_{e(q)}<b\right)$ can be obtained from Corollary 3.1. Here we comment that from (3.31), one can establish (3.23) by using a similar derivation to the proof of Theorem 3.1.
\end{Remark}

\section{A conjecture}
In the proof of Corollary 3.1, for given $q>0$ and $-\gamma < Re(\phi) <0$, we have shown that (see (3.28))
\[
\int_{-\infty}^{\infty}e^{-\phi(x-b)}\mathbb E_x\left[\int_{0}^{e(q)}\textbf{1}_{\{U_s < b\}}ds\right]dx
=-\frac{1}{q \phi}\tilde{\mathbb E}\left[e^{\phi\overline{Y}_{e(q)}}\right]
\mathbb E\left[e^{\phi\underline{X}_{e(q)}}\right].
\]
Moreover, using (3.32) and integration by parts, we have
\[
q\phi\int_{-\infty}^{\infty}e^{\phi(b-x)}\mathbb E_x\left[\int_{0}^{e(q)}\textbf{1}_{\{U_s < b\}}ds\right]dx
=\phi\int_{-\infty}^{\infty}e^{\phi(b-x)}\mathbb P_x\left(U_{e(q)}<b\right)dx
=\int_{-\infty}^{\infty}e^{\phi(b-x)}d\left(\mathbb P_x\left(U_{e(q)}<b\right)\right),
\]
where the second equality is due to (3.23), (3.32) and the condition $-\gamma < Re(\phi) <0$.
Thus, for fixed $b \in \mathbb R$, it follows from the last two formulas that
\begin{equation}
-\int_{-\infty}^{\infty}e^{-\phi(x-b)}d\left(\mathbb P_x\left(U_{e(q)}<b\right)\right)=\tilde{\mathbb E}\left[e^{\phi\overline{Y}_{e(q)}}\right]\mathbb E\left[e^{\phi\underline{X}_{e(q)}}\right].
\end{equation}
Note that identity (4.1) can be extended  analytically to $\phi$ with $Re(\phi)=0$. In fact, for $Re(\phi)=0$, from (3.23) and (3.32), some direct computations yield
\begin{footnotesize}
\[
\begin{split}
-\int_{-\infty}^{\infty}e^{-\phi(x-b)}d\left(\mathbb P_x\left(U_{e(q)}<b\right)\right)
&=
\int_{-\infty}^{b}e^{-\phi(x-b)}\frac{\partial }{\partial x}
\left(\sum_{m=1}^{M_q^+}\sum_{i=1}^{M^q_{m}}H^0_{mi}\frac{(b-x)^{i-1}}{(i-1)!}e^{\beta_{m,q}(x-b)}\right)dx\\
&+\int_{b}^{\infty}e^{-\phi(x-b)}\frac{\partial }{\partial x}\left(\sum_{n=1}^{N_q^-}\sum_{i=1}^{N^q_{n}}G^0_{ni}\frac{(x-b)^{i-1}}{(i-1)!}e^{\gamma_{n,q}(b-x)}\right)dx.
\end{split}
\]
\end{footnotesize}
Note that
\begin{small}
\[
\begin{split}
&\int_{-\infty}^{0}e^{-\phi x}\frac{\partial }{\partial x}
\left(x^{i-1}e^{\beta_{m,q}x}\right)dx
=
\int_{-\infty}^{0}e^{-\phi x}\frac{\partial }{\partial x}
\left(\frac{\partial^{i-1}}{\partial \beta^{i-1}}\left(e^{\beta x }\right)_{\beta=\beta_{m,q}}\right)dx\\
&=\int_{-\infty}^{0}e^{-\phi x}\frac{\partial^{i-1}}{\partial \beta^{i-1}}\left(\beta e^{\beta x }\right)_{\beta=\beta_{m,q}}dx
=\frac{\partial^{i-1}}{\partial \beta^{i-1}}\left(1+\frac{\phi}{\beta-\phi}
\right)_{\beta=\beta_{m,q}}.
\end{split}
\]
\end{small}
Similarly,
\[
\int_{0}^{\infty}e^{-\phi x}\frac{\partial }{\partial x}\left((-x)^{i-1}e^{-\gamma_{n,q}x}\right)dx=
\frac{\partial^{i-1}}{\partial \gamma^{i-1}}\left(-1+\frac{\phi}{\gamma+\phi}
\right)_{\gamma=\gamma_{n,q}}.
\]
Therefore, we have
\begin{footnotesize}
\[
\begin{split}
&-\int_{-\infty}^{\infty}e^{-\phi(x-b)}
d\left(\mathbb P_x\left(U_{e(q)}<b\right)\right)\\
&=\sum_{m=1}^{M_q^+}\sum_{i=1}^{M^q_{m}}H^0_{mi}\frac{(-1)^{i-1}}{(i-1)!}\frac{\partial^{i-1}}{\partial \beta^{i-1}}\left(1+\frac{\phi}{\beta-\phi}
\right)_{\beta=\beta_{m,q}}
+\sum_{n=1}^{N_q^-}\sum_{i=1}^{N^q_{n}}G^0_{ni}\frac{(-1)^{i-1}}{(i-1)!}\frac{\partial^{i-1}}{\partial \gamma^{i-1}}\left(-1+\frac{\phi}{\gamma+\phi}
\right)_{\gamma=\gamma_{n,q}}\\
&=\sum_{m=1}^{M_q^+}H^0_{m1}-\sum_{n=1}^{N_q^-}G^0_{n1}+\phi f^0(\phi)-1=\phi f^0(\phi),
\end{split}
\]
\end{footnotesize}where the second equality follows from (3.24) and the final one is due to (3.30). Hence, from (2.10), (2.15), (3.26) and the last formula, we confirm that (4.1) holds for $Re(\phi)=0$.

\begin{Conjecture}
For $Re(\phi)=0$, identity (4.1) holds for a general refracted L\'evy process $U$ and the corresponding driven process $X$.
\end{Conjecture}

\begin{Remark}
To prove Conjecture 1, one needs first to confirm that equation (2.4) exists a unique strong solution $U$ for a general L\'evy process $X$, and then to establish (4.1). As stated in Remark 2 of Kyprianou and Loedden (2010), the difficulty of proving the existence of a strong solution to (2.4) lies with the case that $X$ has paths of unbounded variation without Gaussian component. Of course, one needs to find the conditions under which the equation (2.4) has a unique strong solution. For the second step, as the process $X$ in (2.1) can be used to approximate any other L\'evy process (see Remark 2.1), identity (4.1) for a general  L\'evy process could be established by using an approximating procedure. The proof of this conjecture is beyond the scope of this paper and is left to future research.
\end{Remark}

A very special case of Conjecture 1 is $\alpha =0$ in (2.4), under which the three processes $X$, $Y$ and $U$ are equal. In this case, we have
\[
\begin{split}
-\int_{-\infty}^{\infty}e^{-\phi(x-b)}d\left(\mathbb P_x\left(X_{e(q)}<b\right)\right)
=\mathbb E\left[e^{\phi X_{e(q)}}\right]=\mathbb E\left[e^{\phi\overline{X}_{e(q)}}\right]\mathbb E\left[e^{\phi\underline{X}_{e(q)}}\right],
\end{split}
\]
where the final equality is known as the Wiener-Hopf factorization (see, e.g., Theorem 6.16 in Kyprianou (2006)). In other words, identity (4.1) reduces to the well-known Wiener-Hopf factorization if $\alpha = 0$.

\end{document}